\documentclass[12pt]{amsart}
\usepackage{amsmath,caption2,amssymb,graphicx,subfigure,epsfig,amsthm}
\newtheorem{theorem}{Theorem}[section]

\newtheorem{proposition}[theorem]{Proposition}

\captionstyle{hang}
\theoremstyle{definition}
\newtheorem{definition}[theorem]{Definition}

\theoremstyle{remark}

\newcommand{\RR}{ {\mathbb R} }

\newcommand{\ff}{\varphi}

\newcommand{\cA}{\mathcal{A}}

\newcommand{\kk}{\kappa}

\newcommand{\ee}{\varepsilon}

\begin{document}

\title[$S$-transform: case of vanishing mean]{Multiplication of free random variables and the $S$-transform: the case of vanishing mean}
\author[N. Raj Rao]{N. Raj Rao
$^{(*)}$}
\thanks{$^*$ Research supported by an Office of Naval Research Special Post-Doctoral Award under grant N00014-07-1-0269.}
\address{Massachusetts Institute of Technology, Department of Mathematics,
77 Massachusetts Avenue, Cambridge, MA 02141, USA.} \email{raj@mit.edu}
\author[R. Speicher]{{Roland Speicher}
$^{(\dagger)}$}
\thanks{$^\dagger\,$Research supported by Discovery and LSI grants from NSERC (Canada) and by
a Killam Fellowship from the Canada Council for the Arts}
\address{Queen's University, Department of Mathematics and Statistics,
Jeffery Hall, Kingston, ON, K7L 3N6, Canada} \email{speicher@mast.queensu.ca}

\begin{abstract}
This note extends Voiculescu's $S$-transform  based analytical machinery for free multiplicative convolution to the case where the mean of the probability measures vanishes. We show that with the right interpretation of the $S$-transform in the case of vanishing mean, the usual formula makes perfectly good sense.
\end{abstract}
\maketitle

\section{Introduction}

Multiplicative free convolution $\boxtimes$ was introduced by Voiculescu
\cite{voiculescu87a}, as an operation on probability measures to describe the
multiplication of free random variables. Since various classes of random matrices become
asymptotically free, the multiplicative free convolution is an important concept to deal
with the asymptotic eigenvalue distribution of products of random matrices
\cite{voiculescu92a,hiai00a}. A powerful analytical tool for an effective calculation of
$\boxtimes$ is Voiculescu's $S$-transform. In the usual presentations of the
$S$-transform it looks as if its definition breaks down in the case where the considered
random variable has mean zero. Since interesting cases, like the multiplicative free
convolution of a semicircle with a free Poisson distribution (corresponding in the random
matrix language to the product of a Gaussian random matrix with an independent Wishart
matrix), fall into this class one gets the impression that the $S$-transform machinery
has a crucial flaw here. However, as we want to point out, this is only ostensive, and
with the right interpretation of the $S$-transform in the case of vanishing mean, the
usual formula makes perfectly good sense. We will show this by combinatorial arguments in
the setting where all moments exists. It stands to reason that one should be able to
establish this using purely analytical tools, for more general distributions for which
the combinatorial moment based arguments do not suffice.

\section{Main result: one mean is zero}

Let us first recall the definition of the $S$-transform. We start by working on an
algebraic level, further below we will address the question of positivity (which is
crucial for the definition of $\boxtimes$). So a random variable $x$ is an element in
some unital algebra $\cA$, which is equipped with a linear functional $\ff$ such that
$\ff(1)=1$. Our main interest is in the moments $\ff(x^n)$ of the random variable $x$. Of
course, of main interest for us are real-valued random variables $X:\Omega\to\RR$ in the
genuine sense, where $\ff$ is given by taking the expectation with respect to the
underlying probability measure,
$$\ff(X^n)=\int_\Omega X(\omega)^ndP(\omega).$$
In our algebraic frame we have to restrict to situations where moments determine the
distribution uniquely; let us, however, remark that by analytic tools much of the theory of
multiplicative free convolution can be extended to classes of probability measures with no
assumption on the existence of moments.

\begin{definition}
Let $x$ be a random variable with $\ff(x)\not=0$. Then its \emph{$S$-transform} $S_x$ is
defined as follows. Let $\chi$ denote the inverse under composition of the series
$$\psi(z):=\sum_{n=1}^\infty \ff(x^n)z^n,$$
then
$$S_x(z):=\chi(z)\cdot\frac{1+z}z.$$
\end{definition}

The above definition has to be understood on the level of formal power series. Note that
$\ff(x)\not=0$ ensures that the inverse of $\psi$ exists as formal power series in $z$.
Furthermore, $\psi$ can be recovered from $S_x$ and thus the $S$-transform contains all
information about the moments of the considered random variable.

The relevance of the $S$-transform in free probability theory is due to the following theorem
of Voiculescu \cite{voiculescu87a}.

\begin{theorem}\label{thm:S-eins}
If $x$ and $y$ are free random variables such that $\ff(x)\not=0$ and $\ff(y)\not=0$,
then we have
$$S_{xy}(z)=S_x(z)\cdot S_y(z).$$
\end{theorem}

Note that in the above setting $\ff(xy)=\ff(x)\ff(y)\not=0$, thus the $S$-transform of
$xy$ is also well-defined and the above theorem allows to get the moments of $xy$ out of
the moments of $x$ and the moments of $y$.

Note that in this theorem we do not require $x$ or $y$ to be real or positive random
variables. In this formulation, the theorem is true in full generality, since it is
essentially a statement about mixed moments of free random variables. However, if one
wants to extract a probability distribution out of the $S$-transform, then one has to
make special requirements concerning $x$ and $y$. In particular, one would like to start
with a situation where $\ff(x^n)$ are the moments of a probability measure $\mu$ on $\RR$
(i.e., $x$ should be a real-valued random variable, i.e., a selfadjoint operator) and
where $\ff(y^n)$ are the moments of a probability measure $\nu$ on $\RR$ (i.e., $y$
should be real-valued random variable or selfadjoint operator). Since $x$ and $y$ do not
commute, $xy$ will not be a selfadjoint operator, thus it is not clear (and in general
will not be the case) that $\ff((xy)^n)$ are the moments of a probability measure.
However, if $x$ is a positive random variable, then $\sqrt x$ makes sense and the
operator $\sqrt x y \sqrt x$ is selfadjoint and has the same moments as $xy$. Thus in
this case we know that the $\ff((xy)^n)$ are the moments of a probability measure on
$\RR$; this is called the \emph{multiplicative free convolution of $\mu$ and $\nu$} and
denoted by $\mu\boxtimes \nu$. So the operation $\mu,\nu\mapsto\mu\boxtimes\nu$ is
defined for probability measures $\mu$, $\nu$ on $\RR$ if at least one of them is
supported on the positive real line.

Often one restricts both factors to be supported in $\RR_+$. In this case the above theorem for
the $S$-transform gives a solution to the problem of calculating $\mu\boxtimes\nu$. But there
are also many interesting cases where only one factor is supported on $\RR_+$. In this case the
description of $\mu\boxtimes\nu$ in terms of the $S$-transform seems to have the following
flaw. In order to be able to invert the power series $\psi(z)=\sum_{n=1}^\infty \ff(x^n)z^n$
one needs a non-vanishing linear term, which means that the mean $\ff(x)$ of the variable
should be non-zero. For measures supported on $\RR_+$ this is (with the exception of the
uninteresting case $x=0$) satisfied, but for measures supported on $\RR$ the mean $\ff(x)$
might be zero and in such a case it seems that the $S$-transform cannot be used to calculate
$\mu\boxtimes\nu$. One should note that this is not an artificial situation, but covers such
common cases like taking the multiplicative free convolution of a semicircle with a
distribution supported on $\RR_+$. One could of course try to approximate mean zero situations
with non mean zero ones; however, the main purpose of our note is to point out that actually
the failure of the $S$-transform description in the mean zero case is just ostensive; with the
right interpretation, the above theorem can also be used in that case to determine
$\mu\boxtimes\nu$.

So let us consider the situation that $\ff(x)=0$ for some selfadjoint element. Let us first see
whether we still can make sense out of the relation $S_{xy}(z)=S_x(z)\cdot S_y(z)$ and whether
this still allows to recover all moments of $xy$ uniquely out of the moments of $x$ and of $y$.
Again we exclude the uninteresting case that $x=0$; hence we know that $\ff(x^2)>0$ (otherwise,
by positivity, $x$ would be equal to $0$) and thus our series $\psi(z)$ starts with a multiple
of $z^2$. This means that although it cannot be inverted by a power series in $z$ it can be
inverted by a power series in $\sqrt z$. This inverse is not unique, but there are actually two
choices (which correspond to choosing a branch of $\sqrt z$); however, this ambiguity does not
affect the final result for the moments of $xy$, if $y$ has non-zero mean. Let us be more
specific about this in the following proposition.

\begin{proposition}\label{prop:eins}
{\rm (1)} Let $\psi$ be a formal power series of the form
$$\psi(z)=\sum_{n=2}^\infty \alpha_nz^n=\alpha_2 z^2+\alpha_3 z^3+\cdots$$
with $\alpha_2\not=0$. Then there exist exactly two power series in $\sqrt z$ which
satisfy
$$\psi(\chi(z))=z.$$
If we denote these solutions by
$$\chi(z)=\sum_{k=1}^\infty \beta_k z^{k/2}=\beta_1\sqrt z+\beta_2 z+\beta_3 {\sqrt z\,}^{3}+\cdots$$
and
$$\tilde\chi(z)=\sum_{k=1}^\infty \tilde\beta_k z^{k/2}=\tilde\beta_1\sqrt z
+\tilde\beta_2 z+\tilde\beta_3 {\sqrt z\,}^{3}+\cdots$$
then their coefficients are
related by
$$\tilde\beta_k=(-1)^k\beta_k.$$

{\rm (2)} Let $\psi$ be as above and consider the two corresponding $S$-transforms
$$S(z)=\chi(z)\cdot \frac{1+z}z$$
and
$$\tilde S(z)=\tilde\chi(z)\cdot\frac{1+z}z$$
Then $S(z)$ and $\tilde S(z)$ are of the form
$$S(z)=\gamma_{-1} \frac 1{\sqrt z}+\sum_{k=0}\gamma_k z^{k/2}$$
and
$$\tilde S(z)=\tilde \gamma_{-1} \frac 1{\sqrt z}+\sum_{k=0}\tilde\gamma_k z^{k/2},$$
where
$$\tilde \gamma_k=(-1)^k\gamma_k.$$

{\rm (3)} Let $x$ be a random variable with vanishing mean and denote by $S_x$ and
$\tilde S_x$ the corresponding two $S$-transforms, as constructed in (1) and (2). Let $y$
be another random variable, with non-vanishing mean so that its $S$-transform is of the
form
$$S_y(z)=\sum_{k=0}^\infty \delta_k z^k=\delta_0 + \delta_1 z+\delta_2 z^2+\cdots$$
Then the two series $S_{xy}(z):=S_x(z)S_y(z)$ and $\tilde S_{xy}(z):=\tilde S_x(z)
S_y(z)$ are of the form
$$S_{xy}(z)=\ee_{-1} \frac 1{\sqrt z}+\sum_{k=0}\ee_k z^{k/2}$$
and
$$\tilde S_{xy}(z)=\tilde\ee_{-1} \frac 1{\sqrt z}+\sum_{k=0}\tilde\ee_k z^{k/2}$$
and related by
$$\tilde \ee_k=(-1)^k\ee_k.$$
If we put
$$\chi_{xy}(z)=S_{xy}(z)\cdot \frac z{1+z},\qquad
\tilde\chi_{xy}(z)=\tilde S_{xy}(z)\cdot \frac z{1+z}$$ and denote by $\psi_{xy}$ the
unique solution of $$\psi_{xy}(\chi_{xy}(z))=z$$ and by $\tilde \psi_{xy}$ the unique
solution of $$\tilde\psi_{xy}(\tilde \chi_{xy}(z))=z,$$ then we have
$$\psi_{xy}(z)=\tilde\psi_{xy}(z).$$

\end{proposition}

\begin{proof}
(1) The equation $\psi(\chi(z))=z$ means
$$z=\sum_{n=2}^\infty
\alpha_n \bigl(\sum_{k=1}^\infty \beta_k z^{k/2}\bigr)^n=
\sum_{n=2}^\infty\sum_{k_1,\dots,k_n=1}^\infty \alpha_n \beta_{k_1}\cdots\beta_{k_n}
z^{(k_1+\cdots+k_n)/2}.$$ By equating the coefficients of powers of $\sqrt z$ this is
equivalent to the system of equations
$$1=\alpha_2\beta_1^2$$
and
$$0=\sum_{n=2}^r\sum_{k_1,\dots,k_n=1\atop k_1+\cdots+k_n=r}^r \alpha_n \beta_{k_1}\cdots\beta_{k_n}$$
for all $r>2$.

The first equation has the two solutions
$$\beta_1=\frac 1{\sqrt\alpha_2}$$
and
$$\tilde \beta_1=-\frac 1{\sqrt\alpha_2}=-\beta_1.$$
Writing the other equations in the form
$$0=2\alpha_2\beta_{r-1}\beta_1+\sum_{n=2}^r\sum_{k_1,\dots,k_n=1\atop
{k_1+\cdots+k_n=r \atop \cdots}} \alpha_n \beta_{k_1}\cdots\beta_{k_n}$$ (where the
$\cdots$ indicate that we exclude in the second sum the cases $k_1=1,k_2=r-1$ and
$k_1=r-1,k_2=1$) one sees that the values of $\beta_n$ for $n>1$ are recursively
determined by $\beta_1$ and the $\alpha$'s. By induction, it follows that $\tilde
\beta_1=-\beta_1$ results in $\tilde\beta_k=(-1)^k\beta_k$ for all $k$.

(2) This is clear, since we have $\gamma_k=\beta_{k+2}+\beta_k$ for all
$k=-1,0,1,\cdots,$ where we set $\beta_{-1}:=\beta_0:=0$.

(3) This follows by reverting the arguments from the first and second part.

\end{proof}

\begin{definition}\label{def:S}
Let $x$ be a random variable with $\ff(x)=0$ and $\ff(x^2)\not=0$. Then its two
\emph{$S$-transforms} $S_x$ and $\tilde S_x$ are defined as follows. Let $\chi$ and
$\tilde \chi$ denote the two inverses under composition of the series
$$\psi(z):=\sum_{n=1}^\infty \ff(x^n)z^n=\ff(x^2)z^2 +\ff(x^3)z^3+\cdots,$$
then
$$S_x(z):=\chi(z)\cdot\frac{1+z}z\qquad\text{and}\qquad
\tilde S_x(z):=\tilde \chi(z)\cdot\frac{1+z}z.$$ Both $S_x$ and $\tilde S_x$ are formal
series in $\sqrt z$ of the form
$$\gamma_{-1} \frac
1{\sqrt z}+\sum_{k=0}^\infty\gamma_k z^{k/2}$$
\end{definition}

Now we stand ready to formulate our main theorem

\begin{theorem}\label{thm:S}
Let $x$ and $y$ be free random variables such that $\ff(x)=0$, $\ff(x^2)\not=0$ and
$\ff(y)\not=0$. By $S_x$ and $\tilde S_x$ we denote the two $S$-transforms of $x$.  Then
$$S_{xy}(z)=S_x(z)\cdot S_y(z)\qquad\text{and}\qquad
\tilde S_{xy}(z)=\tilde S_x(z)\cdot S_y(z)$$ are the two $S$-transforms of $xy$.
\end{theorem}

Note that $xy$ falls into the realm of Def. \ref{def:S}. We have
$$\ff(xy)=\ff(x)\ff(y)=0$$
and
$$\ff((xy)^2)=\ff(x^2)\ff(y)^2+\ff(x)^2\ff(y^2)-\ff(x)^2\ff(y)^2=\ff(x^2)\ff(y)^2\not=0.$$

Let us point out again that the ambiguity of having two $S$-transforms for $xy$ is not a
problem, because both of them will lead to the same moments for $xy$, according to Prop.
\ref{prop:eins}.

In order to prove our Theorem \ref{thm:S}, one might first inspect the usual proofs for
the $S$-transfom. One notices that all of them rely on the fact that $\ff(x)$ and
$\ff(y)$ both are not zero. None of the published proofs can be used directly for the
case where one variable is centered. However, at least the idea of the proof in
\cite{nica97a,speicher:book} can be adapted to our situation. The following proof is a
variant of the approach in  \cite{nica97a,speicher:book} (by avoiding an explicit use of
the incomplete boxed convolution) and can of course also be used to give a
straightforward proof in the usual case $\ff(x)\not=0$.

\begin{proof}

For a random variable $x$ we denote by
$$\psi_x(z):=\sum_{n=1}^\infty \ff(x^n)z^n$$
and
$$M_x(z):=1+\psi_x(z)$$
the corresponding moments series and by
$$C_x(z):=\sum_{n=1}^\infty\kk^x_nz^n$$
the corresponding cumulants series ($\kk_n^x=\kk_n(x,\dots,x)$ is here the $n$-th free
cumulant of $x$).

Consider now $x$ and $y$ as in the theorem. Then, in addition to the moment and cumulant series
of $x$, $y$, and $xy$, we consider also additional moments series of the form
$$M_1(x):=\sum_{n=0}^\infty \ff\bigl(y(xy)^n\bigr)z^n$$
and
$$M_2(z):=\sum_{n=0}^\infty \ff\bigl(x(yx)^n\bigr)z^n.$$
By the moment cumulant formula and by the fact that because of the freeness between $x$ and $y$
mixed cumulants in $x$ and $y$ vanish (see \cite{speicher:book} for the combinatorial theory of freeness),
it is quite straightforward to derive the following relations between these power series:
\begin{equation}\label{eq:eins}
M_{xy}(z)=M_{yx}(z)=C_y[zM_2(z)]+1
\end{equation}
\begin{equation}\label{eq:zwei}
M_1(z)=C_y[zM_2(z)]\cdot\frac{M_{xy}(z)}{zM_2(z)}
\end{equation}
\begin{equation}\label{eq:drei}
M_2(z)=C_x[zM_1(z)]\cdot\frac{M_{xy}(z)}{zM_1(z)}
\end{equation}
Note that all these relations are valid (and make sense as formal power series)
independent of whether $\ff(x)=0$ or not.

Let us note that we have the well-known relation $M(z)=1+C[zM(z)]$ between moment and
cumulant series, which shows, by replacing $z$ with $\chi(z)$, that $C[zS(z)]=z$. This
relation, which was observed in \cite{nica97a}, can be taken as an alternate definition
of $S(z)$. In the case where $\ff(y)\not=0$ this is just a relation between formal power
series in $z$. In the case $\ff(x)=0$, it has again to be read as formal power series in
$\sqrt z$. (Note that for the case $\ff(x)=0$, $\ff(x^2)\not=0$ we also have
$\kk_1^x=\ff(x)=0$ and $\kk_2^x=\ff(x^2)\not=0$, thus $C_x(z)$ starts also with a
quadratic term in $z$.)

Since $\ff(xy)=\ff(x)\ff(y)=0$, the moment series $\psi_{xy}$ has two inverses; let
$\chi$ be one of them. Replacing $z$ by $\chi(z)$ in \eqref{eq:eins} gives
$$1+z=M_{xy}(\chi(z))=C_y[\chi(z)M_2(\chi(z))]+1,$$
yielding that
\begin{equation}\label{eq:vier}
z=C_y[\chi(z)M_2(\chi(z))].
\end{equation}

Equation \eqref{eq:zwei} gives
$$zM_1(z)M_2(z)=C_y[zM_2(z)]\cdot M_{xy}(z).$$
Replacing $z$ by $\chi(z)$ gives
\begin{align*}
\chi(z)\cdot M_1\bigl(\chi(z)\bigr)\cdot M_2\bigl(\chi(z)\bigr)
&=C_y[\chi(z)M_2(\chi(z))]\cdot M_{xy}\bigl(\chi(z)\bigr)
\end{align*}
and hence
\begin{equation}\label{eq:funf}
\chi(z)\cdot M_1\bigl(\chi(z)\bigr)\cdot M_2\bigl(\chi(z)\bigr)=z\cdot(z+1)
\end{equation}

Since $zS_y(z)$ is the unique inverse of $C_y$ under composition we must have, by
\eqref{eq:vier} that
$$zS_y(z)=\chi(z)M_2(\chi(z)).$$

In the same way as for \eqref{eq:vier} we get
$$z=C_x[\chi(z)M_1(\chi(z))].$$
Since $C_x$ has exactly two power series in $\sqrt z$ as inverses under composition,
namely $zS_x(z)$ and $z\tilde S_x(z)$ we must have either
$$zS_x(z)=\chi(z)M_1(\chi(z))\qquad\text{or}\qquad
z\tilde S_x(z)=\chi(z)M_1(\chi(z)).
$$
Name $S_x$ and $\tilde S_x$ in such a way that we have the first equation.

Plugging this into \eqref{eq:funf} gives
$$zS_x(z)\cdot zS_y(z)=\chi(z)M_1(\chi(z))\cdot \chi(z)M_2(\chi(z))=\chi(z) z (z+1)
$$
and thus
$$S_x(z) S_y(z)=\chi(z)\frac {z+1}z=S_{xy}(z).$$
If we start with the other inverse, $\tilde \chi$, of $\psi_{xy}$, then we would end with the
other $S$-transform, $\tilde S_{xy}$.
\end{proof}

This proposition tells us that we can also use the formula $S_{xy}(z)=S_x(z)S_y(z)$ to
calculate the moment series of $xy$, for $x$ and $y$ free, in the case where at most one of
them has non-vanishing mean. Since this is the case in the situation where $y$ has a
distribution supported on $\RR_+$, we see that the $S$-transform is a useful tool for the
calculation of $\mu\boxtimes\nu$, whenever $\mu$ is a probability measure on $\RR$ and $\nu$ is
a probability measure on $\RR_+$, independent of whether $\mu$ has vanishing mean or not.

\section{Case where both means are zero}
One should also note that for the proof of our proposition it is important that at least
one of the two involved measures has non-vanishing mean. If both have vanishing mean then
the arguments break down. One should, however, note that this cannot be attributed to the
fact that there does not exist a probability measure with the moments of $xy$. Actually,
for $x,y$ free with $\ff(x)=0=\ff(y)$ we have by the definition of freeness that
$\ff((xy)^n)=0$ for all $n$, i.e., $xy$ has the same moments as the delta distribution
$\delta_0$. So one might say that the multiplicative free convolution of any two measures
with mean zero is $\delta_0$. However, on a formal level the situation for the
$S$-transform gets quite different here. As a concrete example, let us consider the
multiplicative free convolution of the semicircle with some $\nu$. The $S$-transform of
the semicircle $\mu$ (of variance 1) is
$$S_\mu(z)=\frac 1{\sqrt z}.$$
If we multiply this with an $S$-transform which is a power series in $z$ then the result
is again a series of the form
$$\frac 1{\sqrt z} \cdot \text{power series in $\sqrt z$},$$
which can, as in our proposition, be uniquely resolved for the corresponding moment
series. Let us take now, on the other hand, the multiplicative convolution of two
semicirculars. Then we have
$$S_\mu(z)S_\mu(z)=\frac 1z$$
which results in a $\psi$ transform of the form
$$\psi(z)={\frac 1z-1}.$$
This, however, is not a valid moment series. On the other hand, $\delta_0$ would have the
moment series $\psi(z)=0$, for which no corresponding $S$-transform exists. So we see
that in the case where we multiply two free operators with mean zero the $S$-transform
machinery breaks totally down.

\section{Some examples}

Consider the free multiplicative convolution of the semi-circle distribution with the free Poisson distribution. The semi-circle distribution has $S$-transform
\[
S_{\mu}(z)=\frac{1}{\sqrt{z}}
\]
while the free Poisson distribution has $S$-transform
\[
S_{\gamma}(z)=\frac{1}{z+1}.
\]
The distribution obtained by their free multiplicative convolution has $S$-transform
\[
S_{\mu \boxtimes \gamma}(z)=\frac{1}{\sqrt{z}(z+1)}.
\]
The Cauchy transform of the probability measure $\mu \boxtimes \gamma$  satisfies the algebraic equation
\[
g^4\,z^2-zg+1=0,
\]
from which we can obtain the density function, plotted in Figure \ref{fig:wishtimeswig}.

Consider now the free multiplicative convolution of the free Poisson distribution with the shifted (by $-\alpha$) free Poisson distribution. The free Poisson distribution (shifted by $-\alpha$) has $S$-transform
\[
S_{\mu}^{\alpha}(z) = {\frac {-z-1+\alpha+\sqrt {{z}^{2}+2\,z+2\,z \alpha+1-2\,\alpha+{\alpha}^{2}}}{2z\alpha}}
\]
which, for $\alpha = 1$ yields
\[
S_{\mu}(z) := S^{1}_{\mu}(z) = \frac{-z+\sqrt{{z}^{2}+4\,z}}{2z}
\]
The distribution obtained by the free multiplicative convolution of the free Poisson distribution with the free Poisson distribution (shifted by $-1$) has $S$-transform
\[
S_{\mu \boxtimes \gamma}(z)= \frac {-z+\sqrt {{z}^{2}+4\,z}}{2z \left( 1+z \right)}.
\]
The Cauchy transform of the probability measure $\mu \boxtimes \gamma$ satisfies the algebraic equation
\[
{g}^{4}{z}^{2}+{z}^{2}{g}^{3}-z{g}^{2}-gz+1 = 0
\]
from which we can obtain the density function, plotted in Figure \ref{fig:wishtimeswish}.

We note that the computations in this section were done using \texttt{RMTool} software \cite{raj:rmtool} based on the algebraic equation based framework for computational free probability developed in \cite{raj04a,raj:thesis}.

\begin{figure}
\centering
\subfigure[The solid line is the density function of the probability measure obtained by the free multiplicative convolution of the semicircle with the free Poisson distribution. The histogram bars are the eigenvalues of the product of a Wigner matrix with a Wishart matrix and were generated using $50 \times 50$ sized random matrices over $4000$ Monte-Carlo trials.]{
\includegraphics[width=4.9in]{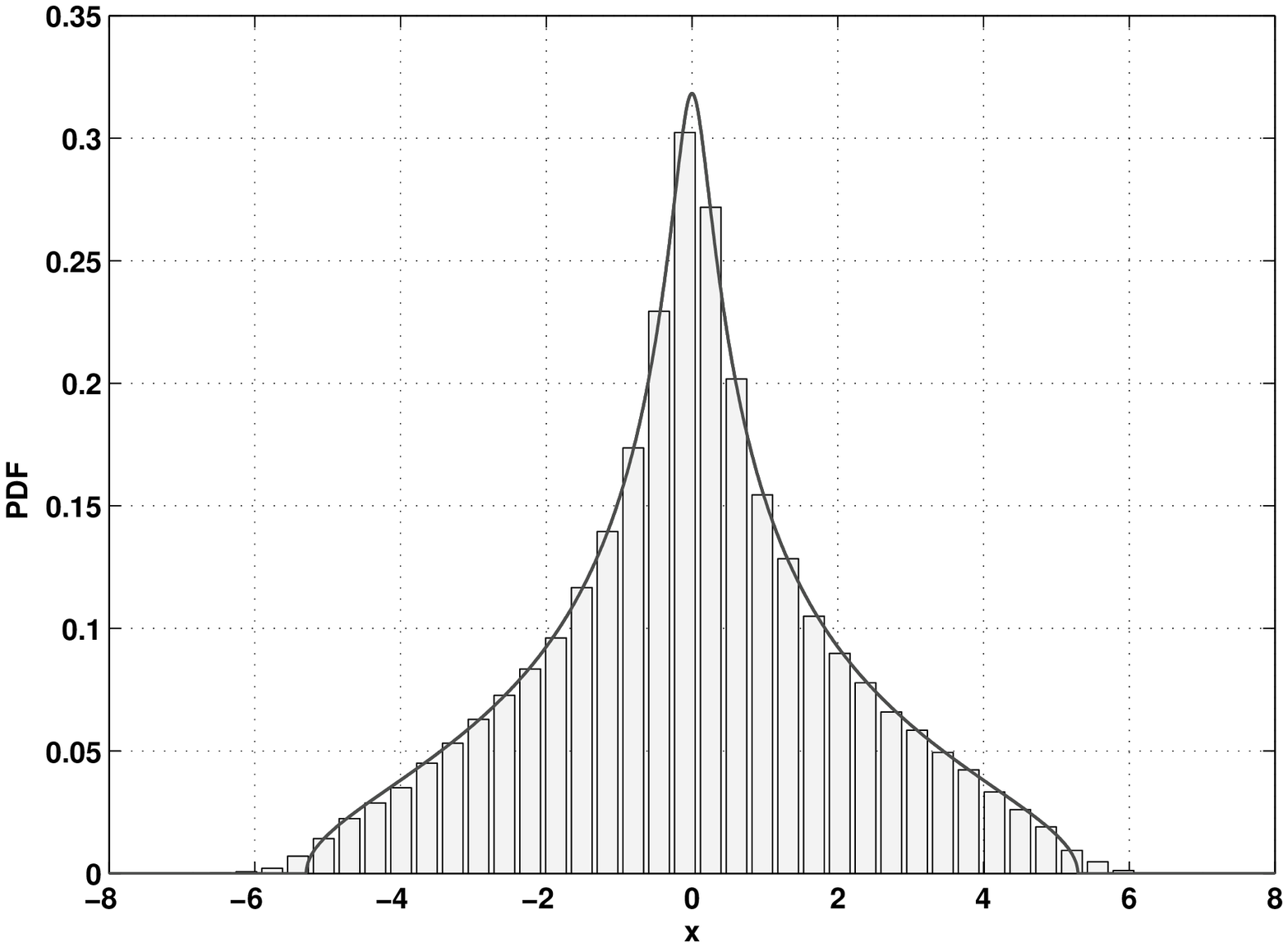}
\label{fig:wishtimeswig}
}\\
\subfigure[The solid line is the density function of the probability measure obtained by the free multiplicative convolution of the free Poisson distribution with the shifted (by $-1$) free Poisson distribution. The histogram bars are the eigenvalues of the product of a Wishart matrix with an independent Wishart matrix (shifted by the negative identity) and were generated using $50 \times 50$ sized random matrices over $4000$ Monte-Carlo trials.]{
\includegraphics[width=4.9in]{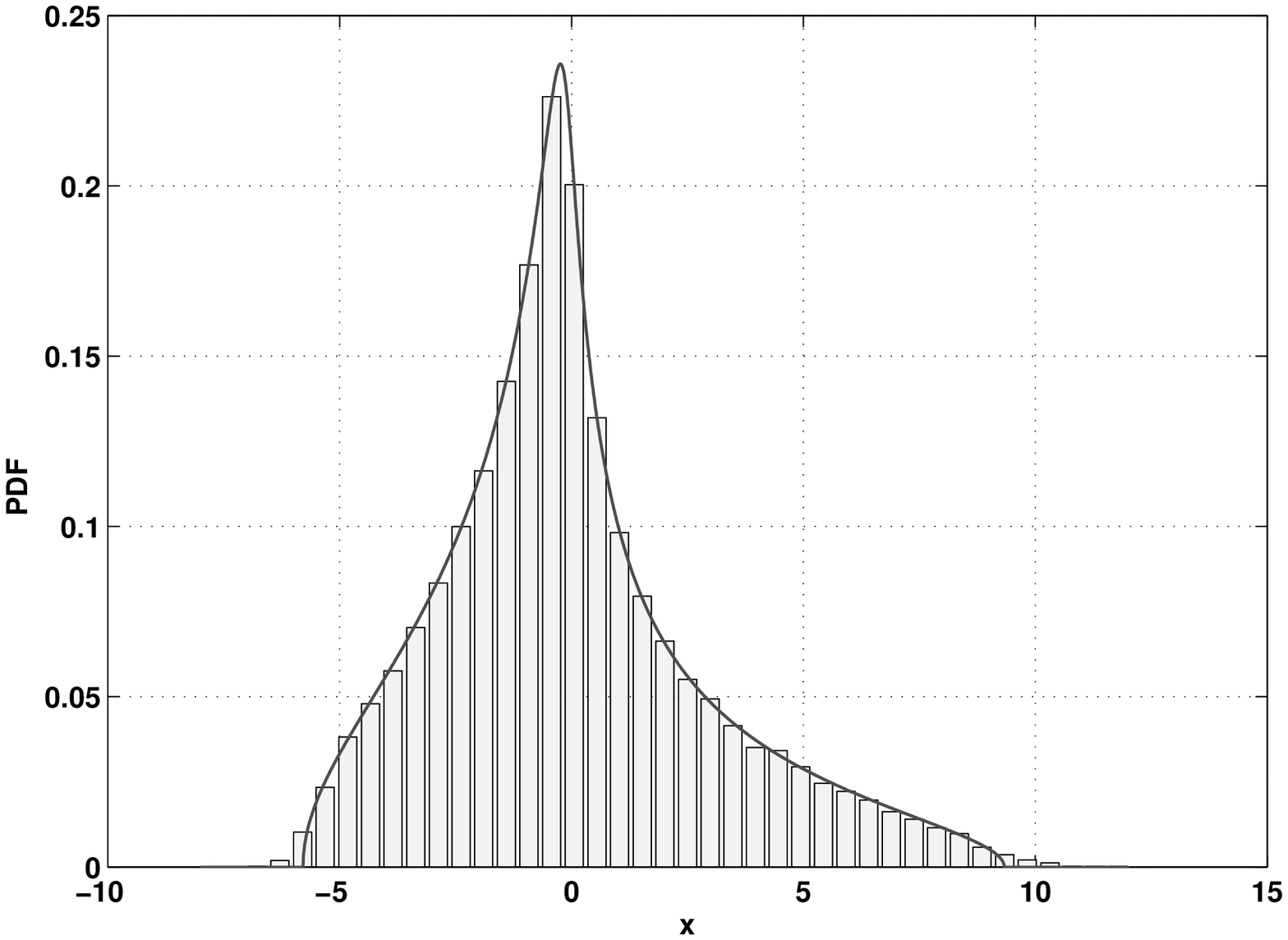}
\label{fig:wishtimeswish}
}
\caption{Examples of free multiplicative convolution with vanishing mean.}
\label{fig:examples}
\end{figure}

\section*{Acknowledgements}
The idea for this paper originated during the authors' correspondence  in the weeks
leading up to March, 2005 Workshop on Free Probability held in Oberwolfach, Germany. We
are grateful to Mathematisches Forschungsinstitut Oberwolfach and the organizers of the
Free Probability Workshop for providing the impetus to interact and their hospitality
during our subsequent stay there.

\bibliographystyle{siam}
\bibliography{randbib}


\end{document}